\documentclass[12pt, oneside,reqno]{amsart}   	
\usepackage{geometry, ulem}
\usepackage{float}
\allowdisplaybreaks
\usepackage{graphicx}				
\usepackage{amssymb}
\usepackage{amsmath}
\usepackage[all]{xy}
\usepackage{mathtools}
\usepackage{tikz-cd}
\usepackage{enumitem}
\usepackage{tikz-cd}
\usepackage{appendix}
\usepackage[T1]{fontenc}
\usepackage{esint}

\usepackage[%
bookmarks=true,
colorlinks,
linkcolor=blue,
urlcolor=blue,
citecolor=blue,
plainpages=false,
pdfpagelabels,
final,
breaklinks=true\between
]{hyperref}
\usepackage{hyperref}
\usepackage[numbers,sort&compress]{natbib}

\newtheorem{thm}{Theorem}[section]
\newtheorem{lem}[thm]{Lemma}

\newtheorem{rem}{Remark}[section]

\newtheorem{defn}{Definition}[section]

\numberwithin{equation}{section}

\def\Pb{\ifmmode{\Bbb P}\else{$\Bbb P$}\fi}
\def\Z{\ifmmode{\Bbb Z}\else{$\Bbb Z$}\fi}
\def\C{\ifmmode{\Bbb C}\else{$\Bbb C$}\fi}
\def\R{\ifmmode{\Bbb R}\else{$\Bbb R$}\fi}
\def\S{\ifmmode{S^2}\else{$S^2$}\fi}

\def\S{\cal S}

\newenvironment{pf}{\paragraph{Proof:}}{\hfill$\square$ \newline}

\begin{document}

\title[On eternal flows]{A classification result for eternal mean convex flows of finite total curvature type} 
\begin{abstract} In this article we partially classify the space of eternal mean convex flows in $\R^3$ of finite total curvature type, a condition implied by finite total curvature. In particular we show that topologically nonplanar ones must flow out of a catenoid in a natural sense. \end{abstract}
\author{Alexander Mramor}
\address{Department of Mathematics, University of Oklahoma, Norman, OK 73019, USA}
\email{amramor@ou.edu}

\maketitle

\section{Introduction}

Considering that in a certain vague sense the mean curvature flow, as the analogue of the heat equation in submanifold geometry, should improve the geometry of its initial data one might expect that the assumption of a flow being eternal should be very strong. Still, the space of eternal mean curvature flows is very large, including minimal surfaces and translators neither space of which is classified completely, so for a more conclusive result more assumptions are needed. With extra assumptions including on the mean convexity and total curvature of the flow our result is:
\begin{thm}\label{mainthm} Let $M_t \subset \R^3$, $t \in \R$, be a complete, properly embedded, connected eternal mean curvature flow of finite entropy such that:
\begin{enumerate}
 \item (Strict) mean convexity: $H > 0$ along $M_t$. 
 \item $M_t$ is of finite total curvature type, in the sense that:
 \begin{enumerate} 
 \item $M_t$ is homeomorphic to a compact surface with a finite number of punctures. 
 \item There exists a uniform constant $C > 0$ such that for each $R > 0$, there exists $t_R < 0$ so that $\int_{B(0,R) \cap M_s} |A|^2 < C$ for all $s < t_R$. 
 \end{enumerate}
  \end{enumerate}
 Then either $M_t$ is an annulus and flows out of a catenoid, in that as $t \to -\infty$ $M_t$ converges to a catenoid from the outside, or is homeomorphic to $\R^2$. 
\end{thm}  
By convergence to a catenoid ``from the outside'' we mean specifically that $\lim\limits_{t \to -\infty} M_t$ is a catenoid $C$ and that for all times $M_t$ lays in the complement of the solid cylinder $C$ bounds. The convergence will be uniform in smooth topology (with multiplicity one) on compact domains. The finite entropy assumption is equivalent to uniform polynomial volume growth and in the context of the mean curvature flow is a fairly ubiquitous background assumption -- here in particular we mean it's bounded independent of time. By \cite{W2}, indeed any type of assumption on total curvature along $M_t$ that ensures it's of finite total curvature for $t$ sufficiently negative (such as a uniform bound, for instance) implies it is of ``finite total curvature type'' as described above  -- minimality is not needed in giving condition (a). 
$\medskip$

As is well known, in the mean curvature flow mean convexity is a very powerful assumption and under the stronger condition of $\alpha$--noncollapsedness such eternal, hence ancient, flows must be convex, and for these there are a number of very strong classification results (for instance \cite{CHH}). Of course we do not assume noncollapsedness in this work. The finite total curvature type assumption simplifies the topology of the flows in question as well as aids in applying compactness results and understanding the associated limit. Such an eternal mean curvature flow out of a catenoid as indicated above, dubbed the ``reapernoid,'' was constructed earlier in \cite{MP} for all $n \geq 2$. Indeed, considering the construction in Theorem 1.3 of \cite{MP} the mean convexity of the eternal solution is obvious and the finite total curvature type assumption above also follows easily. Note that the statement above, unfortunately, does not say anything about the uniqueness of the flow out of the catenoid and imaginably there could be several; note under some symmetry conditions in \cite{MP} we prove some partial uniqueness statements in the same spirit as the present one, but here we are in a more general case. Nor does it say anything about the geometry of $M_t$ in the case it is homeomorphic to the plane. These issues and more are discussed in the concluding remarks below.  
$\medskip$

Concerning ancient flows, it was also shown in \cite{MP} that there are ancient mean convex flows out of all asymptotically flat minimal hypersurfaces, in particular out of all minimal surfaces of finite total curvature because the ends are asymptotically planar or catenoidal. These flows are readily seen to have finite total curvature type in the sense above as well -- we note this result was recently generalized in \cite{Han}. In particular the space of all ancient mean convex flows of finite total curvature type is as complicated as the space as finite total curvature minimal surfaces. This space, as the Costa--Hoffmann--Meeks examples show \cite{Cos, HM1}, is very large so that in this setting the assumption of a flow being eternal is much stronger than merely being ancient. 
$\medskip$

Now we give a sketch of the argument, where throughout $M_t$ satisfies the assumptions of the theorem. First, one can see that the topology of $M_t$ is a punctured sphere. To see this, denoting the region $M_t$ flows into by $H_t$, which is well defined by mean convexity, we first obrserve that $\pi_1(H_t^c)$ is trivial by White's theory for mean convex flows \cite{W}, namely because it says that the forward time limit of the flow is either empty or a collection of disjoint flat planes. As a consequence, we then show that $H_t^c$ is given by a ball with some number of solid half cylinders glued on it. If there is only one solid half cylinder attached then $M_t$ is topologically a plane, although of course there may be more than one such cylinder attached. 
$\medskip$

In the second case, where $M_t$ has more than one end, one can see $\lim\limits_{t \to -\infty} M_t = N$ is nonempty and is a (a priori possibly disconnected) minimal surface of finite total curvature, which by the halfspace theorem we see either is connected and nonflat or is the disjoint union of a number of flat planes. By a result of Lopez and Ros in the first case $M_{-\infty}$ must be a catenoid. The case $M_{-\infty}$ is a disjoint union of planes can be ruled out by a quick topological argument. Again, the other case is that $M_t$ is topologically a plane. We do not further classify the eternal flows in this case (and this space is certainly nonempty) but discuss some hints that, under a stronger true finite total curvature condition, there might in fact be none in the concluding remarks.
$\medskip$

$\textbf{Acknowledgements:}$ The author thanks Niels Martin M{\o}ller for his interest and encouragement in this work. In the course of preparing the article, he was supported by CPH-GEOTOP-DNRF151 from the Danish National Research Foundation, CF21-0680 from the Carlsberg Foundation (via GeoTop and N.M. M{\o}ller respectively) and is grateful for their assistance. The author also thanks the
anonymous referees for their thoughtful reviews whose questions and critique helped
improve the exposition of the article.

 \section{Preliminaries on the mean curvature flow}
In this section we give a nonexhaustive account of some of the results and notions in this paper we will use below, primarily focusing on those concerning the flow. We start with the most basic definition:  
 \begin{defn} A (smooth) mean curvature flow of embedded hypersurfaces in $\R^{n+1}$ is given by a manifold $M^n$ and a family of embeddings $F: M \times I \to \R^{n+1}$ satisfying
 \begin{equation}
 \frac{dF(x,t)}{dt} = -H \nu
 \end{equation}
where $I \subseteq \R$ is some nonempty interval.
\end{defn}
 By a small abuse of notation one often denotes the image of $F$ at a fixed time $t$ by $M_t$. In the following we discuss some especially pertinent properties of the flow, a more complete introduction can be found in \cite{Mant}. 
 $\medskip$
 
  A flow is \textbf{eternal} if $I = \R$, and \textbf{ancient} when at least $I = (-\infty, T]$ for some $T > -\infty$. Under fairly loose hypotheses the mean curvature flow exists for at least a short time from a given initial smooth hypersurface, for instance compactness or uniformly bounded curvature, but producing ancient (or eternal) mean curvature flows is more delicate one reason being that singularities occur often along the flow. 
 $\medskip$
 
 One natural way they appear is in the analysis of these singularities; we say that a flow $M_t$ develops a singularity at a point $(p, t)$ when $|A|^2(p_i, t_i) \to \infty$ for a sequence $(p_i, t_i) \to (p, T)$. To study these singularities one parabolically rescales the flow along such a sequence by factors $\lambda_i \to 0$ and, when appropriate weak assumptions are met to employ compactness theorems are satisfied (see the discussion of Brakke flows below, for example) one will find in a subsequential limit an ancient, possibly eternal, limiting flow which can be thought of as models of the singularity. Of course, one would expect to naturally encounter mean convex ancient or eternal flows as models of singularities along a mean convex mean curvature flow -- although oftentimes (such as for compact mean convex initial data) such flows will be noncollapsed so the singularities will in fact be convex. See \cite{HK0,HK, CHH} for more about the noncollapsing assumption and its consequences. The (nontrivial) eternal singularity models one finds in these cases will be bowl solitons which are topologically planar and in particular none flow out of a catenoid. On the other hand as mentioned above the mean curvature flow is, as well known, the natural analogue of the heat equation so in this sense ancient and eternal flows, as analogues of complete solutions to the heat equation, are arguably interesting regardless of whatever convexity condition or soliton equation they may satisfy or not.
$\medskip$

It is a consequence of the \textbf{comparison principle}, which says that two disjoint hypersurfaces where at least one of them is compact must stay disjoint under the mean curvature flow, that singularities occur often. In particular, in comparison with the bowl soliton \textbf{there can be no compact eternal flows}. Due to the ubiquity of singularities weak notions of the flow are crucial in many situations, and there are a number of ways to define them. A particularly important notion of weak flow is the Brakke flow, a geometric measure theoretic definition which we'll refer to at times in the argument below. These are defined as follows:

\begin{defn} A (n-dimensional integral) Brakke flow is a family of Radon measures $\mu_t$ such that:
\begin{enumerate}
\item For almost every $t \in I$ there exists an integral $n$-dimensional varifold $V(t)$ so that $V(t)$ has locally bounded first variation and has mean curvature vector $\vec{H}$ orthogonal to Tan$(V(t), \cdot)$ a.e. 
\item For a bounded interval $[t_1, t_2] \subset I$ and any compact set $K$,
\begin{equation}
\int_{t_1}^{t_2} \int_K (1 + H^2) d\mu_t dt < \infty
\end{equation}
\item (Brakke inequality) For all compactly supported nonnegative test functions $\phi$, 
\begin{equation}
\int_{V(t_0)} \phi d\mu_{t_0}  \leq \int_{V(0)} \phi d\mu_0 + \int_0^{t_0} \int_{V(t)} \phi H^2 - H \langle \nabla \phi, \nu \rangle - \frac{d\phi}{dt} d\mu_t dt
\end{equation} 
\end{enumerate}
\end{defn}

The study of Brakke flows itself is quite rich, a comprehensive introduction to Brakke flows can be found in \cite{Ton}, but let's briefly discuss a couple facts which will be needed about them. First is \textbf{Brakke's compactness theorem}, which says that for a sequence of Brakke flows with uniform area bounds on parabolic cylinders one may exact a converging subsequence. A second important fact is \textbf{Brakke's regularity theorem}, which says that for Brakke flows with density bounds in a backwards parabolic neighborhood sufficiently close to 1 will be smooth with bounded curvature in a smaller neighborhood. There are related facts for varifolds, \textbf{Allard compactness} and \textbf{Allard regularity} respectively, under analogous assumptions which we will also refer to; see \cite{LS}. Although in this paper we only consider smooth flows, because our assumptions generally aren't strong enough to employ, for instance, Arzela--Ascoli compactness we need to consider these wider classes of objects. 
$\medskip$

The \textbf{entropy} $\lambda(M_t)$, introduced by Colding and Minicozzi in \cite{CM}, of a flow $M_t$ is defined as the supremum over recentered and rescaled Gaussian functionals:
\begin{equation}
\lambda(M_t) = \sup\limits_{x\in \R^3,r\in(0,\infty)} (4 \pi r)^{-n/2} \int_\Sigma e^{-|x - x_0|/4r} d\mu(x)
\end{equation}
This is a quantity that has been of intense recent interest and has many useful properties, but the only point in stipulating the flows $M_t$ in theorem \ref{mainthm} above have bounded entropy is that it provides a compact means of writing that $M_t$ has time independent area bounds to employ the compactness results mentioned above and is a commonly made background assumption. 
$\medskip$

As mentioned in the introduction, the flow is well behaved under a mean convexity assumption, that is $H > 0$. From its evolution equation, 
 \begin{equation}
\frac{dH}{dt} = \Delta H + |A|^2 H
 \end{equation}
one sees by the maximum principle it is typically preserved -- at least in the compact case and for noncompact ones under some extra assumptions as well (in our setting we are just assuming mean convexity at all times). The mean convex mean curvature flow has been considered by many authors. Particularly important in our applications is the work of B. White \cite{W}. We will use a number of his results from this article, a result one can see from sections 10, 11 of his paper we wish to highlight is the following:
\begin{thm}\label{White_thm1} Suppose that $M_t \subset \R^{3}$ is a mean convex mean curvature flow defined on an interval $(T, \infty)$, where $T < \infty$. Then $\lim\limits_{t \to \infty} M_t$ is either empty or a stable minimal hypersurface $M_\infty$, possibly with multiple components.  The convergence is smooth away from the singular set of $M_\infty$. 
\end{thm}
Of course since we are considering minimal surfaces in $\R^3$ one can further conclude the convergence above will be smooth, although it may be with multiplicity. In particular, the one sided minimization property of mean convex mean curvature flows gives that the limiting set, if nonempty, will be smooth and stable within any bounded domain which says the limiting set in totality is stable.
$\medskip$

Lastly, a useful way to characterize mean convex flows is in terms of the corresponding evolution of the sets they bound -- because $M_t$ is mean convex it bounds a set $H_t$ which, in terms of set inclusion, is monotonically decreasing as $t$ increases. Strictly so in fact, because $H$ is assumed to be strictly positive. Because $M_t$ is assumed connected it and its complement will be connected as well. We will refer to this characterization often in the arguments below.

\section{Proof of theorem}

In this section, $M_t$ satisfies the assumptions of theorem \ref{mainthm}. The first lemma follows quickly from the discussion above. 
\begin{lem}\label{longterm} The set $\lim\limits_{t \to \infty} M_t$ is either empty or diffeomorphic to a disjoint union of planes. In the latter case the convergence is smooth in bounded domains with multiplicity one. 
\end{lem}
\begin{pf}
By White's theorem, theorem \ref{White_thm1}, from the section above we know the limit set, which we'll denote here by $M_\infty$, will be either empty or a (possibly disconnected) smooth and stable minimal surface -- considering each component at a time without loss of generality it is connected. If it is empty we are finished so suppose $M_\infty \neq \emptyset$. Since the flow is mean convex we see that $P \subset H_t$ for all $t \in \R$, and because $M_t$ is connected then the flow must lay on one side of $P$. The multiplicity bound theorem of White \cite{W} then says that $M_t$ must converge to $P$ with multiplicity one. Since $M_t$ embedded $M_\infty$ is as well, and it will be closed so complete. Using it is stable and complete then by \cite{FCS} it must be a plane $P$.
\end{pf} 

Using this lemma we can already say a lot about the topology of $M_t$, via its bounded domain(s) -- because $M_t$ is connected there are exactly two connected components, where one must be $H_t$ and the other $H_t^c$. First we show:
\begin{lem}\label{fun_triv} $\pi_1(H_t^c)$ is trivial.
\end{lem}
\begin{pf}
{To begin note that since the flow is smooth for any $s, t \in \R$ $M_s$ and $M_t$ are isotopic, and because it is a proper flow the isotopy extension theorem then implies that the pairs $H_s, H_t$ and $H_s^c, H_t^c$ are isotopic (and thus homotopic) where, as we recall, $H_t$ denotes the set $M_t$ bounds and flows monotonically into. Now, for the sake of concreteness we consider the fate of the flow starting from time $t = 0$,} and {we} suppose that there is a homotopically nontrivial curve $\gamma \in H_0^c$; note by the mean convexity of the flow that $H_t^c \subset H_s^c$ for $t < s$ so for $t > 0$ it makes sense to consider $\gamma \subset H_t^c$. {Since from above $H_0^c$ and $H_t^c$ are homotopic for any finte $t > 0$, $\gamma$ is then homotopically nontrivial in $H_t^c$ for any finite $t > 0$}. On the other hand by lemma \ref{longterm} it must be the case that $\gamma$ is homotopically trivial in $H_\infty^c$, or so that it bounds an immersed disc in $H_\infty^c$. Of course the disc lays in a compact set, so because by lemma \ref{longterm} $H_t^c$ converges to $H_\infty^c$ which is uniform in compact sets after perhaps perturbing the disc slightly it must be contained in $H_t^c$ for some potentially large but finite $t$ though, giving a contradiction because the flow is smooth. 
\end{pf}
$\medskip$   

 Recall that by assumption $M_t$ is homeomoprhic to a punctured compact surface $\widetilde{M}$, fixed for any time $t$ since the flow is by assumption smooth. a priori, $M_t$ could have many ends (or, $\widetilde{M}$ could have many punctures), possibly arranged in complicated ways. First we show a result that can be interpreted as saying they are arranged in a simple way; note that when we write ``solid half cylinder'' we mean specifically a domain of $\R^3$ bounded by a capped off half cylinder (i.e. a surface homeomorphic to a plane), and do not further claim this domain is homeomorphic to the domain bounded by a capped off standard round cylinder persay -- the trepidation here due to the fact the asymptotic structure of the ends is not prescribed in the assumptions. {Soon after though we will show such domains are simply connected, which will suffice for our needs}. 
 $\medskip$
 
{Below for the sake of concreteness we fix $t =0$ once again;
 to interpret the following statement we recall in the special case of asymptotically conical surfaces, which have a well defined link $L \subset S^2$, that ends are nested precisely if not all components of the link bound discs in $S^2 \setminus L$, which in turn corresponds to the ends not bounding solid cylinders (in our notation, in either $H_0$ or $H_0^c$). With this in mind one concrete interpretation of the following is that there aren't any nontrivially "nested" ends:}
 $\medskip$
 
\begin{lem}\label{handlebody} $H_0^c$ is homeomorphic to a smooth bounded domain with a finite number of solid half cylinders glued along it. 
\end{lem}

\begin{pf} 
By (a) of the finite total curvature type assumption $M_0 \simeq \Sigma \setminus \{p_1, \ldots, p_k\}$, where $\Sigma$ is a compact surface. By ends of $M_0$ we specifically mean a choice of disjoint annuli $B(p_i, r_i) \setminus p_i$ for some small $r_i > 0$, under the homeomorphism above. Denote by $\eta_1, \ldots \eta_k$ the (image of the) boundary curves of these annuli, and the annuli they bound in $\R^3$ by $C_1, \ldots, C_k$. Then by lemma \ref{fun_triv} each of these curves is homotopically trivial in $H_0^c$, and since they are embedded they in fact bounded embedded discs $D_1, \ldots, D_k$ in $H_0^c$ by Dehn's lemma. By standard cutting and pasting arguments these discs can be taken to be disjoint. So, combined with the cylinder $C_1, \ldots, C_k$ they give disjoint capped off cylinders laying in $H_0^c$ (partially on its boundary, of course) bounded by the ends and discs. Because $M_0$ is connected by assumption these cylinders must bound solid cylinders in $H_0^c$, giving the claim. 
\end{pf}
$\medskip$

Of course, lemma \ref{fun_triv} strongly suggests that the compact domain above should be a ball. This is indeed the case, as we show next: 
$\medskip$

\begin{lem}\label{fun_triv2}$H_0^c$ is homeomorphic to a ball with a finite number of solid half cylinders glued along it so in particular, referring to part (a) of the definition of finite total curvature type, $\widetilde{M}$ must be a sphere. 
\end{lem}
$\medskip$

\begin{pf} 
Suppose that the genus $g$ of $\widetilde{M}$ is greater than or equal to 1. Again fixing the time $t = 0$ to be concrete, by capping off the ends of $M_0$ into the solid half cylinders of $H_0^c$ using the disjoint discs from the lemma above and smoothing out the edges we obtain a smooth surface $\Sigma$, homeomorphic to $\widetilde{M}$, which bounds a domain $U$. Because the capped ends are all homeomorphic to planes, one can see by Seifert--Van Kampen that the domains in $H_0^c$ they, {that is the capped off ends,} bound are indeed simply connected {as the term suggests}.
$\medskip$

{To explain this step in more detail denoting one of these capped ends by $N$, homeomorphic to $\R^2$ be the previous lemma, we consider the domains $V_1, V_2$ it bounds in $\R^3$. By slightly fattening these domains $V_1, V_2$ we get a cover of $\R^3$ whose intersection deformation retracts onto $N$. Since $N$ is simply connected by Seifert--Van Kampen the fundamental group of $\R^3$ is isomorphic to the free product of the fundamental groups of $V_1$ and $V_2$, which imply these must both be homotopically trivial because $\R^3$ itself is simply connected}. 
$\medskip$

Using this one can continuously deform any closed curve in $H_0^c$ into one laying entirely {in} $U$ so that $\pi_1(U) \simeq \pi_1(H_0^c)$. Since $\overline{U}$ is a compact 3--manifold {with boundary $\Sigma$}, we have that the rank of the image of the map $H^1(\overline{U}) \to H_1(\Sigma)$ is half the rank of $H_1(\Sigma)$ (see lemma 3.5 of \cite{Hat}). When $g \geq 1$ the rank of $H_1(\Sigma)$ is nonzero of course, which implies that $H^1(\overline{U})$ and hence $H_1(\overline{U})$ is nontrivial, as a consequence of the universal coefficient theorem. So, by the Hurwicz isomorophism, neither is $\pi_1(U)$ contradicting lemma \ref{fun_triv}.
\end{pf}

To refine this picture, we will glean information about $M_t$ from its limit in the far past i.e. as $t \to -\infty$. The first lemma says that the limit is well defined/the convergence is full, at least in a weak sense, essentially because the sets $H_t$ are monotone increasing in terms of set inclusion as $t \to -\infty$:  
\begin{lem}\label{Haus_conv} The limits $\lim_{t \to -\infty} H_t$, $\lim_{t \to -\infty} M_t$, which we denote by $H_{-\infty}$, $\partial H_{-\infty}$ respectively, are well defined/the convergence is full in the Hausdorff topology on bounded domains. 
\end{lem} 
\begin{pf} 

First note that the space of compact subsets of a compact set (here, we consider the closure of $H_t$ in $B(0,R)$) is compact so that for any sequence $t_i \to -\infty$ $H_{t_i} \cap B(0,R)$ subconverges to some, possibly empty, set. Now let $t_k, t_\ell \to -\infty$ be two sequences for which $H_{t_k}\cap B(0,R)$ Hausdorff converges to a set $S_1$ and $H_{t_\ell}\cap B(0,R)$ converges to $S_2$. By set monotonicity of the flow it's easy to see that $S_1 \subseteq S_2$ and vice versa, so they are equal implying for any sequence $t_m \to -\infty$ that $H_{t_m}$ must subconverge to a set $S$ independent of sequence. Then the full convergence claim follows by a compactness contradiction argument. As a consequence the boundaries of $H_t$, which are precisely $M_t$, must also fully converge in Hausdorff distance on bounded sets to a set which, by some abuse of notation we'll define as $\partial H_{-\infty}$ (since a set $C$ and its closure are the same in Hausdorff distance, there is potentially some ambiguity when speaking of the boundary of $\partial H_{\infty}$). 
\end{pf}
$\medskip$

We next refine the geometry of this limit; For the sake of exposition first let's show a result in the case $M_t$ has pointwise uniformly bounded curvature, by which we mean $|A|^2(p,t) < C$ for some fixed constant $C$ independent of $p$ and $t$-- this is indeed satisfied by the eternal flow out of a catenoid mentioned in the introduction, but a priori this isn't immediate. Precisely, let's first show: 
\begin{lem}\label{farpast_easy} Supposing in addition to the assumptions of Theorem \ref{mainthm} above that $|A|^2(p,t)$ is bounded by a uniform constant $C$, the set $\lim\limits_{t \to -\infty} M_t$ is either empty or diffeomorphic to a disjoint union of smooth embedded complete minimal surfaces, where the convergence is full in the sense of varifolds for any given fixed bounded domain. 
\end{lem}

\begin{pf}
 Note that by the pointwise curvature bound and Shi's estimates that $|\nabla^k A|$ for any $k \geq 1$ is bounded uniformly along $M_t$ for all times $t < 0$. This gives by the entropy bound and Arzela--Ascoli subsequential convergence of the flows $M_{t + s_i}$ for a sequence $s_i \to -\infty$ in $C^k$ norm (a priori, possibly multisheeted) for any $k \geq 0$ on bounded sets. The limit flow will be smooth (but potentially only immersed) and evolve by the classical mean curvature flow.  Naturally denoting this limit flow by $M_{-\infty,t}$, we see that a priori it may depend on the sequence of $s_i$ but actually because the Hausdorff convergence in lemma \ref{Haus_conv} above is full we see this is not the case and in fact by the assumed curvature bound, namely that in balls of fixed radius $M_t$ can be written as a union of graphs of uniformly bounded gradient, we see the support of $M_{-\infty, t}$ must be precisely $\partial H_{-\infty}$.
$\medskip$

Now, because the $M_t$ are all embedded we see that $\partial H_{-\infty}$, in some suitably small ball $B$, can be locally written as the union of smooth graphs which at worst intersect tangentially. The flow $M_{-\infty, t}$ is also represented locally by the union of these graphs, although one might imagine with different multiplicities. At any rate because the flow $M_{-\infty, t}$, which is smoothly varying, is precisely $\partial H_{-\infty}$ for all times these graphs must each be minimal. Then by the maximum principle these graphs must either agree or be disjoint so $\partial H_{-\infty}$ must actually be an embedded minimal surface. 
$\medskip$

We next claim that for fixed $t$ as $s \to -\infty$ that the convergence of $M_{t + s}$ to $\partial H_{-\infty}$ must be with multiplicity one -- by the curvature bound assumption we see the convergence must be locally graphical but with multiple connected components. To see this, by the set monotonicity of the flow if the backwards convergence was greater than 2 at some point of $\partial H_{-\infty}$ then $M_t$ could not intersect $\partial H_{-\infty}$ and there would have to be a sheet on either side of $\partial H_{-\infty}$, so $\partial H_{-\infty}$ would disconnect $M_t$ giving a contradiction (the flow is assumed to be connected). In particular, for any sequence $t_i \to -\infty$ for which $M_{t_i}$ converges in the sense of varifolds the limit varifold must be the same ($\partial H_{-\infty}$ with multiplicity one). With this in mind it's reasonable to denote then $\partial H_{-\infty} = M_{-\infty}$, and a compactness--contradiction argument (using Allard's compactness theorem via the entropy bound and finite total curvature type assumption) gives that the convergence $M_{t} \to M_{-\infty}$ is full in the sense of varifolds. 
\end{pf}

In fact, because the backwards convergence above is with multiplicity one it is not too hard to see it is smooth and we discuss this more below. To deal with the more general case we employ the finite total curvature type assumption along with Simon's sheeting theorem as employed by Ilmanen in \cite{IP}. $\medskip$

\begin{lem}\label{farpast} The set $\lim\limits_{t \to -\infty} M_t$ is either empty or diffeomorphic to a disjoint union of smooth embedded complete minimal surfaces, where the convergence is locally in the sense of varifolds possibly with multiplicity. 
\end{lem}
\begin{pf}
 By the finite total curvature type condition and entropy bound we may employ Brakke's compactness theorem with the flows $M_{t+s_i}$ where $s_i \to -\infty$ to obtain a subsequential limit Brakke flow which we'll denote here by $M_{-\infty, t}$, a priori dependent on the choice of sequence taken. Of course (as a weaker statement) for $t$ fixed under our assumptions by Allard's compactness theorem for any sequence $s_i \to -\infty$ we may extract a subsequence $s_k \to -\infty$ such that $M_{t+ s_k}$ varifold converges, and it follows from lemma \ref{Haus_conv} for any sequence $s_k \to -\infty$ for which $M_{t + s_k}$ converge as varifolds that the support of the limit varifold $\nu$ is contained in $\partial H_{-\infty}$, which by abuse of terminology as above we define as the Hausdorff limit of $M_t$. Similar to above we want to study next to what extent is the limit $\nu$ dependent on the sequence $s_k$. 
$\medskip$

To proceed it's helpful to have a more concrete handle on the nature of the convergence $M_t$ to a limit varifold. Since we assume $M_t$ has finite entropy and satisfies part (b) of the finite total curvature type condition we draw inspiration from Ilmanen's analysis in \cite{IP} to do so -- see also \cite{Sun} for a nice summary. Sketching this out, considering the measures $\sigma_t = |A|^2 \mathcal{H}^2 \llcorner M_t$, where $\mathcal{H}^2$ is the 2--dimensional Hausdorff measure, by the finite total curvature type condition for a sequence $t + s_k \to -\infty$ we may take a subsequence $t + s_\ell \to -\infty$ for which these radon measures $\sigma_{t + s_\ell}$ converge to a limit measure $\sigma$. By Simon's sheeting theorem in balls $B$ where $\sigma_t$ is sufficiently small (i.e. where the integral curvature doesn't concentrate) we can locally write $M_t \cap B$ as a union of embedded discs with controlled area. With this in mind, suppose that $t + s_k \to -\infty$ is a sequence for which as varifolds $M_t$ converge to some measure $\nu$. Denoting the support of $\nu$ by $N$, then by the finite total curvature type assumption and the sheeting theorem it follows there is a subsequence $t + s_\ell$ as above so that in each ball $B(0,R)$ there will be a finite (possibly empty) set of points $Q = Q(R)$ for which the convergence $M_t \to N$ in $B(0,R)$ will be multisheeted over $N \setminus Q$. In particular, writing the set $Q = \{p_1, \ldots, p_k \}$ we have there are $r_{s_\ell} \to 0$ for which in the complement of $\cup_{i=1}^k B(p_i, r_i(s_\ell))$ $M_{t + s_\ell} \cap B(0,R)$ is the union of disjoint surfaces ``sheeted'' over $N$, which may be connected/bridged within the balls. 
$\medskip$

We next show that the support $N$ of the/a limit measure $\nu$ as described above must in fact equal $\partial H_{-\infty}$ up measure zero inspired by the proof of corollary 5.3 of \cite{Sun}. Supposing this is not the case, then of course if we further pass to the subsequence $M_{t + s_\ell}$ defined above the limit of sequence will also not have support equal to $\partial H_{-\infty}$ (indeed, the limit measure is still $\nu$), so it suffices to consider this sequence to gain a contradiction. Working towards a contradiction, suppose then $p \in \partial H_{-\infty}$ is not in the support $N$ of $\nu$. Then there is a $B(p,r)$ for which $N$ has measure zero in $B(p,r)$. Because $\partial H_{-\infty}$ is the Hausdorff limit of $M_t$ and the convergence is full for $t$ sufficiently large $M_t \cap B(p,r/2)$ is nonempty. Considering all this within a large fixed ball $B(0,R)$, from the paragraph above there is a finite set of points $Q$ away from which the convergence of $M_{t + s_\ell}$ is multisheeted, so in particular for $t + s_\ell$ sufficiently negative there is a finite set of points $\{p_1, \ldots, p_k\}$ so that $M_{t + s_\ell} \cap (\cup_i B(p_i, r/1000))^c$ is given by a union of disjoint embedded surfaces. Because for $t + s_\ell$ sufficiently negative $M_{t+ s_\ell}$ passes through $B(p, r/2)$ we then get a contradiction, because as a consequence of this and of the sheeting theorem (the discs in the sheeting theorem have a lower area bound as well as an upper bound) the measure of $M_{t + s_\ell} \cap B(p,r)$ is uniformly bounded below for very negative $t + s_\ell$ giving a contradiction. 
$\medskip$

 Now, a technical issue is that the multiplicity of convergence here may depend on the sequence $s_k$. In the expository lemma above, lemma \ref{farpast_easy}, we first showed that $\partial H_{-\infty}$ was minimal and then showed the multiplicity of convergence was independent of sequence, but how we will proceed here the steps are reversed and this point on multiplicity we'll soon see is important. Note since $M_t$ is embedded the number of sheets in the sheeting theorem must agree in overlaps of balls where it applies, so since we can apply it away from a discrete set of points the multiplicity of convergence must be constant on connected components of $\partial H_{-\infty}$ for a given sequence $s_k$. By the uniform entropy bound, this number is uniformly bounded above independent of the sequence. Suppose on a connected component of $\partial H_{-\infty}$ the convergence is with multiplicity $a$ along $s_k$ and $b$ along say the sequence $s_k'$ with $a \neq b$. Intertwining the sequences to get a sequence $s_\ell$, we see that by the smoothness of the flow there must be times $t_\ell$ between $t + s_\ell$ and $t + s_{\ell + 1}$ for which the flow is not sheeted away from a set of small balls over $\partial H_{-\infty} \cap B(0,R)$ {for the number of sheets to change between these times; see figure \ref{fig:nonsheeting}}. 
 $\medskip$

{Indeed, restricting ourselves to some connected component of $\partial H_{-\infty}$ in a large ball $B(0, R)$ as above, suppose for the sake of contradiction that on some time interval $I = (-\infty, T]$, $T$ very negative the flow was always sheeted over $\partial H_{-\infty}$, away from some finite number of concentration points. In particular given some time $t \in I$ away from some small balls $B(p_i, r)$ centered at these points $M_t$ would be sheeted over $\partial H_{-\infty}$ within $B(0,R)$. Shrinking $r$ if necessary, we may suppose the doubled balls $B(p_i, 2r)$ are disjoint. Since the dimension $n = 2$ the sets $M_t \cap ( B(p_i, 2r) \setminus B(p_i, r))$ are each a union of say $k_i$ annuli, and because $M_t \setminus B(p_i, 2r)$ is smooth and sheeted and we're working over a connected component of $\partial H_{-\infty}$ we see that each of the $k_i$ are equal (note this is a step which breaks down in one dimension lower because cutting out a small ball/interval on a 1-manifold will potentially disconnect it) and furthermore the multiplicity of sheeting is the same over every point of (the connected component in consideration of) $\partial H_{-\infty} \cap B(0,R)$ away from the bridge points. Because the flow is smooth in time these concentration points could migrate along $\partial H_{-\infty}$ but the number of sheets in the convergence could not change on the interval $I$, giving a contradiction to the presumption it does.}
 $\medskip$

 Taking a subsequence of $M_{t_\ell}$ as above (using the sheeting theorem, etc.) we gain a contradiction, so the multiplicity of convergence must be independent of converging sequence. {As an aside, alternately one may be able to argue using the monotonicity of the entropy: along the sequence $s_k$ the entropy should be approximately  $a \cdot \lambda(\partial H_{-\infty})$ and along the other sequence $s_k'$ approximately $b \cdot \lambda(\partial H_{-\infty})$, for $k$ very large; since entropy is monotone under the flow this would give a contradiction unless $a = b$.} 
 $\medskip$

 \begin{figure}
     \centering
     \includegraphics[width=1\linewidth]{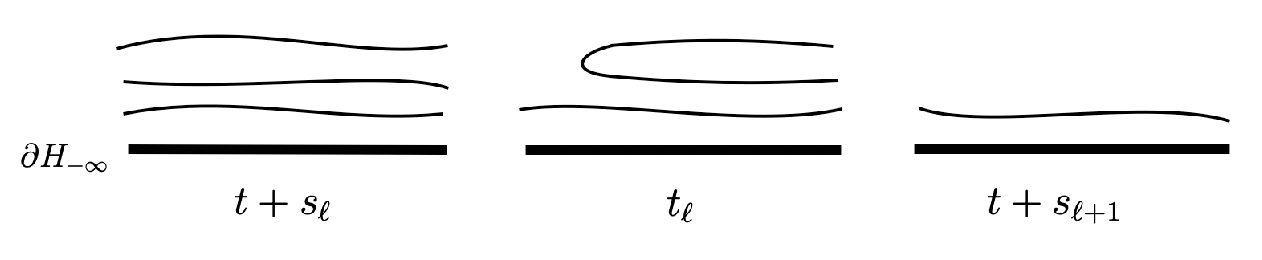}
     \caption{{If the multiplicity of convergence depended on choice of sequence $s_k \to -\infty$ one could find another sequence of times $s_\ell \to -\infty$ for which the multiplicity of convergence to $\partial H_{-\infty}$ alternates (between $3$ and $1$ in the figure above), and by the smoothness of the flow there must then be an intermediate sequence of times $t_\ell \to -\infty$ where the flow cannot be sheeted over $\partial H_{-\infty}$ away from finitely many points. By repeating the sheeting argument along this sequence, this leads to a contradiction.} {Note that the figure above is drawn one dimension lower, and that in our $n=2$ case the set of bad points corresponding to a change in number of sheets as indicated above couldn't be discrete.}}
     \label{fig:nonsheeting}
 \end{figure}
 
As a consequence, by another compactness contradiction argument with Allard's compactness theorem, we can see the varifold convergence of $M_{t+s}$ is full/true in $s \to -\infty$ (again, within any given fixed bounded domain of $\R^3$).  Considering the definition of Brakke flow convergence this gives the Brakke flow $M_{-\infty, t}$ defined above is independent of the sequence $s_i \to -\infty$ chosen, implying $M_{-\infty,t}$ $M_{-\infty, t + \delta}$ are the same for any $\delta$. Hence we find again that $M_{-\infty, t}$ is time independent. As discussed in chapter 5 of \cite{Ton}, this implies that $M_t$ is stationary or so that  $|H| = 0$ a.e. on $M_{-\infty, t}$. We write this stationary set as $M_{-\infty}$ like in the expository lemma. 
$\medskip$

Because $|H| = 0$ a.e. on $M_{-\infty}$, so that it's minimal, when considering the measure convergence on just one of these sheets in a ball where $\sigma$ doesn't concentrate then by applying Allard's regularity theorem we get that $M_{-\infty}$ has $C^{1, \alpha}$, and so smooth by regularity theory for minimal surfaces in balls, support. This reasoning applies in balls with centers in all of $M_{-\infty}$ except potentially at a discrete set of points where the measures concentrate -- these correspond to places where the sheets are potentially bridged. These can be pieced together to realize $M_{-\infty}$ globally as an immersed minimal surface away from the points mentioned. Gulliver's removable singularity theorem then give $M_{-\infty}$ is a smooth minimal surface which can be seen to be embedded because the flow $M_t$ is using the maximum principle. As a result it is also complete.  \end{pf}
$\medskip$

Before using this to say something more about $M_t$ (for finite times) first we refine the topology this convergence is in, where actually the information about the topology of $M_t$ we've already gained comes into play in a minor (possibly inessential) way. Below we strongly use the assumption that $M_t$ is connected: note for embedded minimal surfaces $\Sigma$ of finite total curvature one create ancient flows which propagate into both bounded domains of $\Sigma$, and considering these as one single (disconnected) ancient mean curvature flow they will flow back to $\Sigma$ as $t\to -\infty$ with multiplicity two. 
$\medskip$

\begin{lem}\label{conv_smth} The convergence of $M_t \to M_{-\infty}$ is locally smooth with multiplicity one, and thus the components of $M_{-\infty}$ are smooth complete minimal surfaces of finite total curvature. 
\end{lem}

\begin{pf}
We suppose that $M_{-\infty}$ is nonempty, or else there is nothing to do. We note by Brakke's regularity theorem and Shi's estimates the convergence of $M_t \to M_{-\infty}$ will be in $C^\infty_{loc}$ if we can show the multiplicity of the convergence in the argument above is one by a compactness contradiction argument -- an easy way to see we have the necessary density bounds to apply the regularity theorem is by the one sided minimizing property for mean convex flows (again, \cite{W}) used with small tubular neighborhoods of $M_{-\infty} \cap B(0,R)$. By the finite total curvature type assumption this will also imply $M_{-\infty}$ has bounded total curvature, and as above the limit will be complete. 
$\medskip$

As pointed out in the proof above as a consequence of Simon's sheeting theorem we have that for any sequence $s_i \to -\infty$, after potentially taking a subsequence, most, if not all of $M_{s_i}$ can be written as the union of some number of disjoint sheets over $M_{-\infty}$ in a given ball $B(0,R)$ for $s_i$ sufficiently negative -- by ``most'' here we mean away from a set (for $t$ sufficiently negative) of small disjoint balls within which there may be ``bridges'' connecting the sheets. The number of sheets must be finite by the finite entropy assumption, of course, and the lemma above gives that the multiplicity of convergence is independent of such a sequence. Our main goal is to gain a contradiction if there is such a sequence $s_i \to -\infty$ for which $M_{s_i} \to M_{-\infty}$ with multiplicity $m = 2$ or greater. 
$\medskip$

{At an organizational level, we proceed by first showing that if the multiplicity is greater than 2 there must be bridges joining them (subclaim 1) which are persistent (subclaim 2) and don't wander in the sense that they found at all sufficiently negative times and lay in some bounded region; we will next argue that there are in fact no such bridges (subclaims 3 and 4).} 
$\medskip$

{Since in lemma \ref{fun_triv2} $\widetilde{M} \simeq S^2$ we see that at bridge points $p_j$ $M_{s_i} \cap B(p_j(s_i), r_j(s_i))$ are homeomorphic to a union of discs with a number of discs connected by cylindrical segments (justifying more literally the use of the term ``bridges''). These annuli could have other annuli nested within them, in the obvious sense, or bound solid cylindrical segments in $H_t$ or $H_t^c$ within the balls $ B(p_j(s_i), r_j(s_i))$; see figure \ref{fig:bridgetypes} for an illustration of these cases which we will refer to as solid, nested, or hollow. Of course, the convergence to $M_{-\infty}$ is also in Hausdorff distance in bounded domains so as one looks back in time the curvature along these necks must blow up. In particular with the expository lemma, lemma \ref{farpast_easy}, in mind if we had supposed a uniform curvature bound these would be ruled out immediately, going far enough back in time.}
$\medskip$

\begin{figure}[H]
    \centering
    \includegraphics[width=0.75\linewidth]{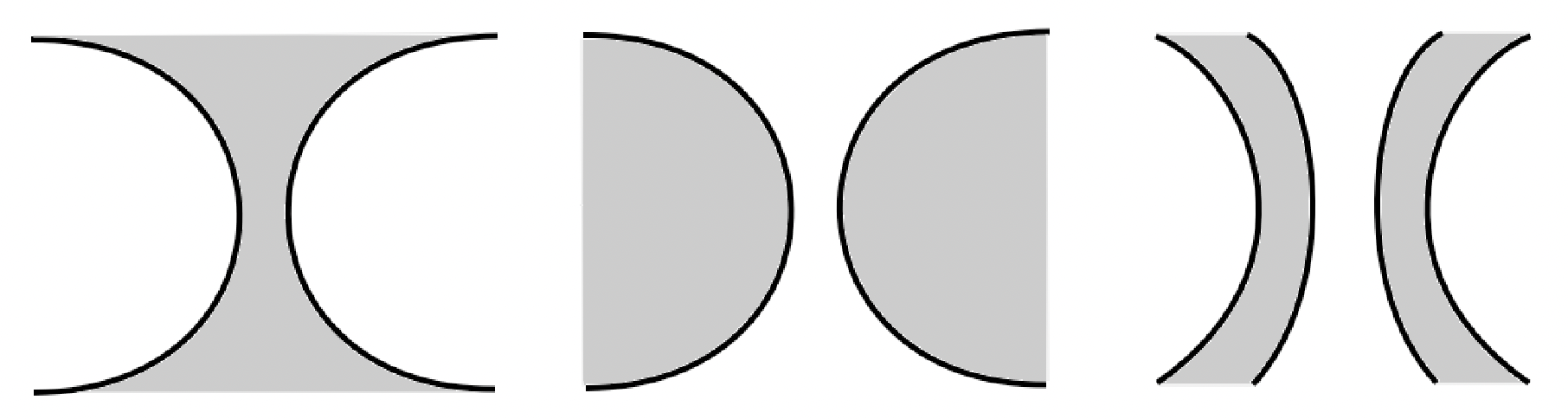}
    \caption{{Where the shaded region indicates $H_t$, from left to right the following are what we'll refer to as solid, hollow, and nested bridges respectively (note in the latter case the configuration could be more complicated than indicated). Very briefly by mean convexity the solid and nested cases are ruled out because otherwise there would be times $t > -\infty$ for which the flow couldn't be smooth, while hollow necks must quickly "open" which easily contradicts the converence $M_t \to M_{-\infty}$.}}
    \label{fig:bridgetypes}
\end{figure}

Because dealing with each of the possible cases is relatively lengthy for the reader's convenience below we break the argument into a number of highlighted subclaims:
$\medskip$

\textbf{Subclaim 1: there must be "bridges" joining sheets.} We first claim there indeed must be bridges connecting at least some of these sheets where $R$ is large enough, corresponding to where the curvature concentrates in the limit -- it isn't immediate there should be any despite $M_t$ being connected because we are in the noncompact case and so a priori where the sheets ``connect/wrap around'' could tend to spatial infinity in the backwards limit, and also where the curvature concentrates there could just be ``pimples'' instead of actual bridges. More precisely what we are claiming is that if $R$ is sufficiently large there are a finite (nonempty) set of balls $B(p_1(s_i), r_1(s_i)), \ldots B(p_{N(s_i)}(s_i), r_{N(s_i)}(s_i))$ in $B(0,R)$ for which ($M_{s_i} \cap B(0,R)) \setminus \{ \cup_{j = 1}^{N(s_i)} B(p_j(s_i), r_j(s_i)) \}$ is a union of $m(s_i)$ disjoint smooth manifolds with boundary, that $M_{s_i} \cap B(0,R)$ has $1 \leq c < m(s_i)$ connected components, and the maximum radius of the balls above tends to zero as $s_i \to -\infty$. Here by $N(s_i)$ and $m(s_i)$ we denote the number of balls $N(s_i)$ and number of sheets $m(s_i)$ respectively. These numbers must eventually stabilize by the choice of sequence, but imaginably it could be the case that $m(s_i) \neq m$ for $s_i$ relatively positive. 
$\medskip$

Now, suppose this isn't the case; for the claim to fail for this sequence it must be for a given $R > 0$ that $M_{s_i} \cap B(0,R)$ has $m(s_i)$ components for $s_i$ sufficiently negative. By the varifold convergence for $s_i$ sufficiently negative $m(s_i) = m$. Using that the flow is smooth arguing exactly as in the lemma above where we showed the multiplicity of convergence was independent of sequence $s_k \to -\infty$ a compactness contradiction argument with the sheeting theorem gives that there will be no ``new'' sheets of $M_{s_i} \cap B(0,R)$ in that sheets of $M_{s_i} \cap B(0,R)$ correspond to components of $M_{s_j} \cap B(0,R)$ for $i,j$ large enough. Specifically, if this is not the case then because the flow is smooth there must be a sequence of times $t_k \to -\infty$ for which the flow isn't locally sheeted over $M_{-\infty}$ (away from shrinking neighborhoods of a finite set of points in the sense described above), but one can extract a subsequence for which this must be true. 
$\medskip$

With this in mind because $H_t, H_t^c$ alternate across these sheets by the (strict) mean convexity of the flow we see that the multiplicity of convergence $m$ must be equal to two; see figure \ref{fig:nobridges}. As indicated in the figure the two connected components of $M_{s_i} \cap B(0,R)$ by set monotonicity would have to locally bound a (domain homeomorphic to a) slab in $H_t^c$, and because we assume strict mean convexity of the flow it must be the case $\partial H_{-\infty}$ lays strictly within.

\begin{figure}
    \centering
    \includegraphics[width=0.75\linewidth]{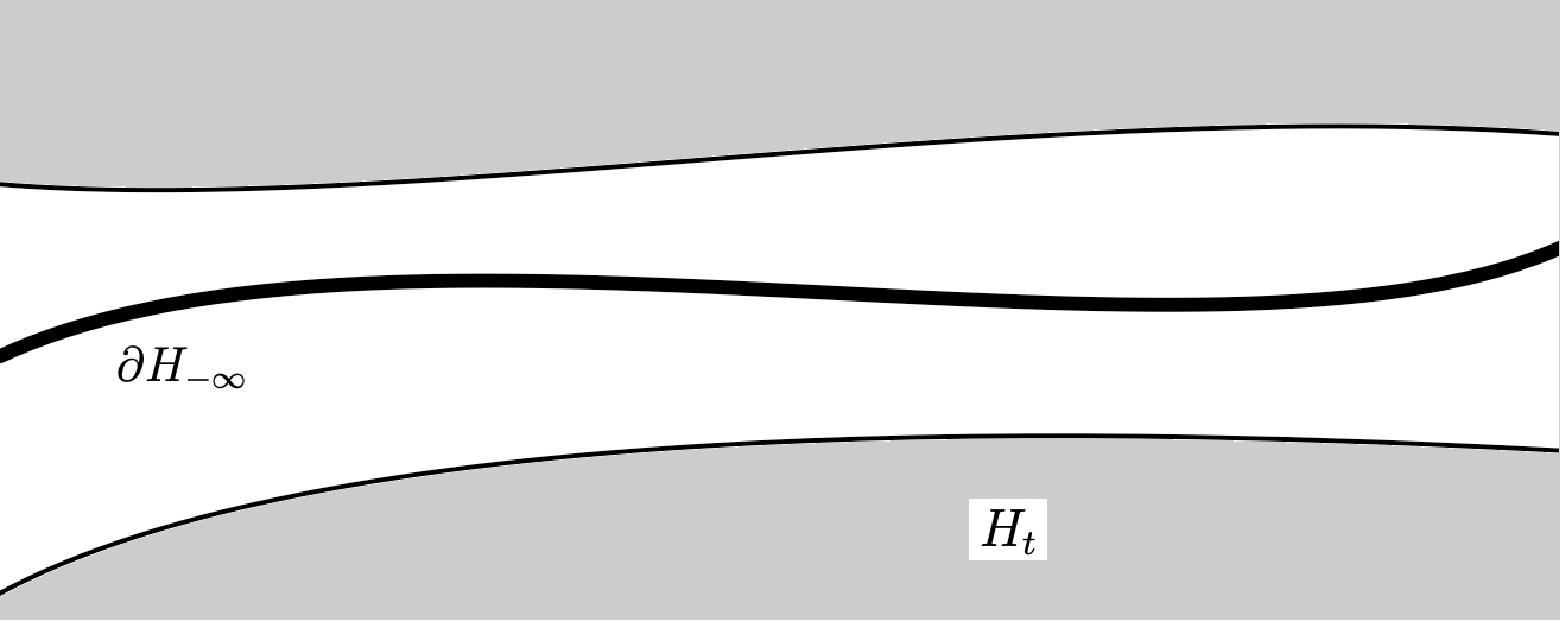}
    \caption{{In the case there are no bridges in the convergence of $M_{s_i}$ to $\partial H_{-\infty}$, the number of sheets in the convergence must be two. This is because the domains $H_t$ and $H_t^c$ alternate across the sheets in this case (within a fixed domain), and so if there are more than 2 sheets some will in fact be "repelled" from $\partial H_{-\infty}$ be the mean convexity of the flow. In this case, because as $t \to -\infty$ $H_{t}^c$ is contracting, locally the configuration must be as illustrated above. From this its not to hard to show that the flow of $M_t$ must be disconnected, contradicting our assumptions on $M_t$ in Theorem \ref{mainthm}.}}
    \label{fig:nobridges}
\end{figure}

$\medskip$

 Now because that are no new sheets along $M_{s_i}$ for $i$ sufficiently {large} we may sensically discuss them in between times $s_i, s_{i+1}$, although it could be the case in between these times they are joined by bridges which may have propagated in from outside $B(0,R)$: we next rule these out {(similar reasoning is then used in the next subclaim)}. Supposing there are such bridges, one can see since $H_t^c$ is a growing domain forwards in time that because at the time $s_i$ there are no bridges in $B(0,R)$ by the set monotonicity of the flow they could only be given by necks which either bound solid cylindrical segments in $H_t$ {or are nested}. Because $H_t$ is an increasing domain backwards in time at least some of the  bridges must persist to time $s_{i+1}$, and such bridges cannot have moved far by set monotonicity of the flow. Considering $R$ slightly larger would then give a contradiction to the presumption the claim fails, giving that $M_{-\infty}\cap B(0,R) \subset H_t^c$ for all sufficiently negative times. Because $M_{-\infty}$ is embedded so disconnects $\R^3$ and $m = 2$ this gives for every $R >>0$ there is $s$ sufficiently negative so that $M_{t} \cap B(0,R)$ is disconnected for $t< s$. Since $H_t^c$ is an expanding domain under the flow (considered forward in time) this then gives that actually $M_t$ must be disconnected, giving a contradiction because we suppose $M_t$ is connected. \hfill $\Box$
$\medskip$

\textbf{Subclaim 2: WLOG, bridges are "persistent" {in the sense that they correspond to bridges for all sufficiently negative times,} and don't wander far {in the sense they individually lay in a bounded region}.} For a bridge found at time $s_i$ as long it joins two sheets we have that for $s_k < s_i$ (indeed for all $t < s_i$) if the bridge laid in $B(p_j(s_i), r_j(s_i))$ it cannot be distance greater than $r_j(s_i)$ from $p_j(s_i)$, because otherwise a portion of $H_{t}^c$ would have to overlap with $H_{s_i}$ for some $t < s_i$ which is impossible considering the domain $H_t$ is increasing backwards under the flow; this is a coarse upper bound but suffices for our purposes as it keeps bridges we may find from sliding to spatial infinity farther back in time. Now a priori there is the possibility of a number of the bridges potentially disappearing as we follow them back in time corresponding to a drop in the number of sheets in the convergence. For instance those bounding solid cylinders segments in $H_t^c$ might bound on one side a ball in $H_t^c$ which may collapse considering the flow further back in time because this is a contracting domain, but if the neck persists it cannot move far. Because we have ruled out the ``no bridges'' case above and they can't travel very far, it is the case that at least some of them do last as $t \to -\infty$ in $B(0,R)$ and without loss of generality we focus our attention on these ``persistent'' bridges, of which there must be at least one because we are supposing the convergence backwards in time is at least two. \hfill $\Box$
$\medskip$

\textbf{Subclaim 3: there are no solid or nested bridges.} With subclaim 2 in hand, to start we can see that none of these cylindrical segments may bound a solid cylindrical segment in $H_t$. Suppose on the contrary one does, say contained in a ball $B(p_j(s_i), r_j(s_i)) \subset B(0,R)$ where $R$ is sufficiently large. First note $\sup |A|^2 \to \infty$ in $M_{s_i} \cap  B(p_j(s_i), r_j(s_i))$ as $s_i \to -\infty$ {using lemma 2.4 in \cite{CM1}}, in particular it must do so along these connected cylindrical segments {which in turn, crucially, contract by mean convexity as $t$ increases. In particular,} considering the flow is mean convex and so into the domain $H_t$, and the bridges contained in $B(p_j(s_i), r_j(s_i))$ are all contained in some large ball, say $B(0,R)$ after {potentially} adjusting $R$, {we see that at sufficiently negative but finite times} that $\sup |A|^2 = \infty$ within $M_{s_i} \cap B(p_j(s_i), r_j(s_i)) \subset B(0,R)$, contradicting the smoothness of the flow {since it is complete}. Alternately, perhaps one could argue using that since these solid cylindrical segments in $H_t$ should be expanding under the flow going back in time by the monotonicity of this set we would reach a contradiction to $M_t$ converging to $M_{-\infty}$. Similarly, we can rule out the appearance of any persistent annuli which have others nested within in $M_{s_i} \cap  B(p_j(s_i), r_j(s_i))$. \hfill $\Box$
$\medskip$

 \textbf{Subclaim 4: there are no hollow bridges.} As a result the only remaining cylindrical segments are those which bound a solid cylindrical segment in $H_t^c$, which one might call "hollow." One can see then that consequently in any large ball $B(0,R)$ for $s_i$ sufficiently negative $M_{s_i}$ is the disjoint union of surfaces given by pairs of sheets bounding regions of $H_t$, with some pairs connected by such segments (so locally {modeled on a slab}, potentially punctured {i.e. a slab with a solid cylinder "drilled" out}). Our first claim in this case is that at most one of these pairs may persist as $s_i \to -\infty$, so the multiplicity of convergence is no more than 2. As in the no bridges case there will be no ``new'' pairs along the sequence $s_i$ for $i$ sufficiently large. Restricting ourselves to such $i$ because $M_t$ is connected for any two adjacent pairs there is a curve $\gamma$ in $M_{s_i}$ joining them, which should necessarily go outside $B(0,R)$. Now if we were not able to pick (in a continuous way) $\gamma_t$, where $\gamma_0 = \gamma$ so that it remained in a fixed bounded region for more negative times it would be the case that for any $R >0$ if $i$ is large enough then $H_{s_i} \cap B(0,R)$ is disconnected -- as in the no bridges case. {With that case and} the configuration of $M_{s_i} \cap B(0,R)$ in mind this would give {as before} that $M_t$ is disconnected giving a contradiction. So we may define $\gamma_t$ so it remains in a bounded region for more negative times, say within $B(0, \widetilde{R})$. Considering that $M_{s_i}$ eventually is sheeted over $M_{-\infty}$ in $B(0,\widetilde{R})$ as well $\gamma$ must go through a bridge far enough back in time, and such a neck by the choice of $\gamma_t$, as it goes between two {regions of $H_t$ bounded by pairs of sheets} of {$M_t$}, must go along (one or a number of) of nested necks or one which bounds a solid cylinder in $H_t$, particularly at least one of which must persist which was ruled out. 
 $\medskip$

 \begin{figure}
     \centering
     \includegraphics[width=0.85\linewidth]{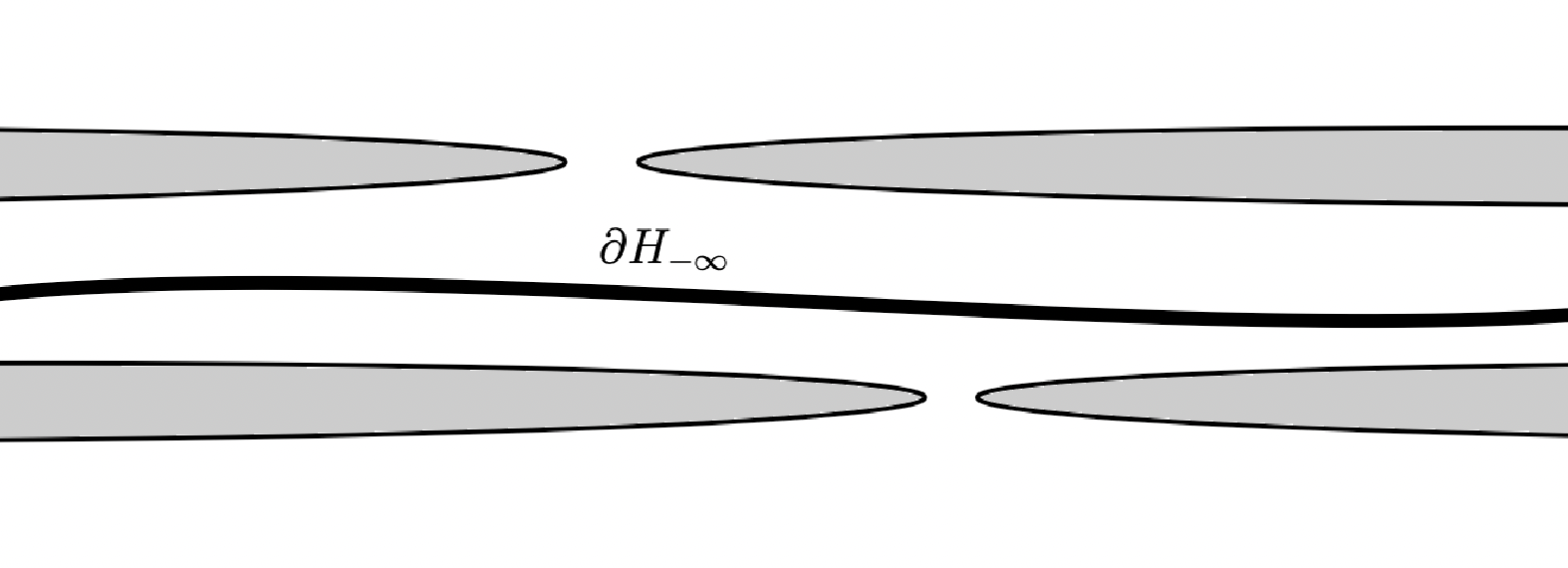}
     \caption{{If there are only hollow bridges, then the sheets of $M_t$ must come in pairs which bound domains in $H_t$, shaded above, as indicated. If there is more than one such pair they must be joined by a curve $\gamma$ by the connectedness assumption, but such a curve would have to cross a solid or nested bridge which was previously ruled out.}}
     \label{fig:nosolid}
 \end{figure}
 
Hence, for $i$ sufficiently large there must be some $j$ for which $H_{s_i} \cap  B(p_j(s_i), r_j(s_i)) \subset B(0,R)$ is a flattened, punctured ball laying in a slab. Again, this bridge cannot slide to infinity tracing it back farther in time, and will always be contained in a large ball (which again we may suppose will be $B(0,R)$). Note if $r_j$ is sufficiently small depending on $M_{-\infty}$ and taking $s$ more negative if necessary, the width of the slab can be taken to be as small as one wishes -- of course it is the case that the radii of all the balls where bridges occur in some fixed $B(0,R)$ tend to zero from earlier in the argument. Considering that the backwards limits $M_t$ and $M_{t + \delta}$ agree as $t \to -\infty$, as discussed in the argument for the previous lemma, then White's expanding hole theorem \cite{W} says that the cylindrical bridges cannot collapse in the backwards limit, giving a contradiction to $M_t$ varifold converging to $M_{-\infty}$. Alternately one can gain a contradiction using the mean convexity of the flow considering the sheets away from the bridges, because the domain $H_t$ increases backwards in time and so one can see at most one sheet could move towards $M_{-\infty}$ as $t$ is further decreased.  \hfill $\Box$ 
$\medskip$

To summarize we showed that if the number of sheets in the backwards convergence was with multiplicity greater than one, then there must be at least one persistent bridge joining them using the set monotonicity of the flow. Again using the set monotonicity of the flow though we could rule out such bridges, implying that  for any sequence $s_i \to -\infty$, a subsequence of $M_{s_i}$ must converge with multiplicity one, and hence smoothly, to $M_{-\infty}$. Because the varifold convergence of $M_t \to M_{-\infty}$ is full this then implies the statement. 
 \end{pf}

We next refine our topological information about $M_t$ using that the backwards limit is a minimal surface of finite total curvature. The main use of this fact is that minimal surfaces of finite total curvature must have parallel ends. For the purposes of showing the main result it is slightly stronger than necessary, and we give an alternate path for the sake of exposition afterwards.

\begin{lem}\label{2ends} {At any given finite time t} the number of ends of $M_t$ is no more than two, and $H_t^c$ is homeomorphic to a ball with either one or two solid half cylinders glued along it. 
\end{lem} 

\begin{pf} 
By the mean convexity assumption we observe the domain $H_t^c$ is decreasing in terms of set inclusion as we consider $t$ becoming more and more negative, or in other words $H_t^c \subset H_0^c$ for all $t$ negative. We note that by using planes as barriers within the solid cylinders of lemma \ref{handlebody} that the number of ends of $M_t$ cannot decrease in finite time as we consider $t \to -\infty$ (if there were such a time, run the flow forward and use collections of planes intersecting the solid half cylinders as barriers to gain a contradiction -- no noncompact maximum principles are required as the cross sections of solid cylinders are compact). {In particular, if the number of ends at some time $t$ of $M_t$ is $k \geq 2$ then this must also be the case for all times preceding it.} 
$\medskip$

{On the other hand,} the bowl soliton shows {the number of ends may decrease} in the limit and in fact the backwards limit can even be empty. If $M_t$ has two or more ends this issue doesn't occur, however; {we next show that if $M_t$ has $k \geq 2$ ends, so must $M_{-\infty}$}. If $M_t$ has two or more ends, we may consider properly embedded lines $\gamma(s,t) \subset M_t$ between any two of them, in that as $s \to -\infty$ $\gamma(s,t)$ lays on one end and as $s \to \infty$ $\gamma(s,t)$ lays on another. As $M_t$ is smooth such curves are defined for all $t$, and by lemma \ref{handlebody} and the mean convexity of $M_t$ any such choice of embedded line between two different ends must intersect $H_0^c \cap B(0, R)$, $R >> 0$, for all times $t<0$ giving that $M_{-\infty}$ is nonempty. Applying a similar argument within each of the solid half cylinders from lemma \ref{handlebody} one can also see that $M_{-\infty}$ will also have at least as many ends as $M_t$ does, although it could conceivably have more. 
$\medskip$

 Continuing to suppose $M_t$ has two or more ends, because $M_{-\infty}$ is a minimal surface of finite total curvature the ends of it are parallel by \cite{JM} (see also \cite{MY}), by which we mean $M_{-\infty}$ is ambiently isotopic to connect sums of parallel planes perhaps along with some compact surface: see the figure below (also, figure 1 in \cite{MY}). This then implies that the ends of $M_t$ are parallel in the same sense. One can then readily see that if the number of ends is strictly greater than 2 then $\pi_1(H_t^c)$ is not trivial, giving a contradiction to lemma \ref{fun_triv2}. 
\end{pf}

\begin{rem} \textbf{(Alternate to lemma above)} We point out that the statement above is actually stronger than we need to prove the theorem; it suffices to know that in the case that $M_t$ has two or more ends then $M_{-\infty}$ is nonempty and has a connected component with at least two ends. To see this, suppose that $M_t$ has two or more ends. Then from lemma \ref{fun_triv2} $\pi_1(H_t)$ must be nontrivial. Considering a homotopically nontrivial curve in $\gamma \subset H_0$, because going backwards in time $H_t$ is an expanding domain it makes sense to consider $\gamma \subset H_t$ for $t < 0$. If $M_{-\infty}$ is either empty or only contains components with one end which must be disjoint planes (since $M_{-\infty}$ has genus zero and finite total curvature) then $\gamma$ bounds a disc $D \subset H_{-\infty}$. After a potential slight perturbation of $D$ away from $M_{-\infty}$ then by the local smooth convergence of $M_t \to M_{-\infty}$ for possibly very negative but finite $t$ $D \subset H_t$, giving a contradiction since the flow is smooth so $\gamma$ should stay homotopically nontrivial. Note the main result does imply the lemma above, a posteori. 
\end{rem}

 \begin{figure}
\centering
\includegraphics[scale = .5]{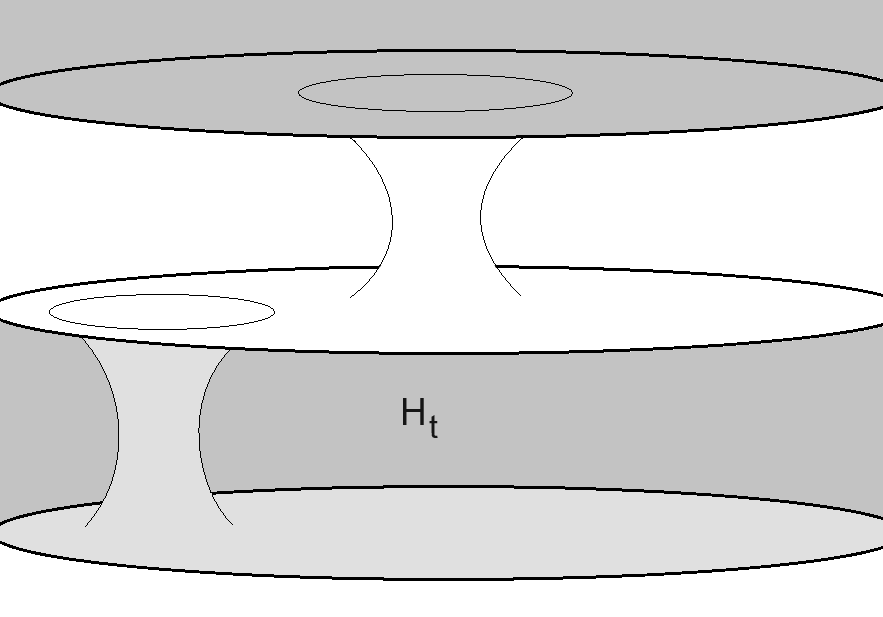}\label{fig:parfig}
\caption{ \small This figure illustrates a potential configuration of $M_t$ if it were to have more than 2 ends. Since the flow is mean convex and flows to either the empty set or union of planes, the top ``neck'' would have to collapse contradicting the smoothness of the flow.}
\end{figure}

\begin{proof}[Proof of theorem \ref{mainthm}] If the number of ends of $M_t$ is two, then from the argument of lemma \ref{2ends} above $M_{-\infty}$ is nonempty and has at least two ends. By the halfspace theorem \cite{HM} either $M_{-\infty}$ is a single nonflat minimal surface or consists of a disjoint union of flat planes. Because $M_{-\infty}$ has finite total curvature it is homeomorphic to a genus $g$ surface with $k$ punctures for its $k$ ends -- note that the number of ends could imaginably increase in the limit so that $k > 2$, otherwise we would be done by, for instance, \cite{Sch}. In fact, as a side remark we point out it can be written in a very concrete way as a compact surface of genus $g$ connected to a sum of parallel planes, up to isotopy (see figure 1 in \cite{MY}). Now, because the convergence of $M_t \to M_{-\infty}$ is smooth in compact domains by lemmas \ref{farpast} and \ref{conv_smth} we immediately get a contradiction if $g \geq 1$ to lemma \ref{fun_triv2}. Since $g = 0$ $M_{-\infty}$ is a punctured sphere so has genus 0 and, by the theorem of Lopez and Ros \cite{LR}, must be a catenoid.  
$\medskip$

In the first, nonflat, case above then by lemma \ref{fun_triv} $M_t$ must flow out of the catenoid in the manner indicated in the introduction, so we next rule out the latter case arguing similar to in the remark above. In the second case we see that $H_{-\infty}$ contains no homotopically nontrivial curves. On the other hand, by the set monotonicity of the flow $H_t \subset H_{-\infty}$ for all $t \in \R$, and in our case by lemma \ref{2ends} $H_t$ does have such curves for all $t$; denote such a curve by $\gamma \subset H_0$ and by $D$ the immersed disc contained in $H_{-\infty}$ it bounds. By lemmas \ref{farpast} and \ref{conv_smth} the convergence of $H_t \to  H_{-\infty}$ is smooth in compact domains, which implies that $D$, up to a slight perturbation, is contained in $H_t$ for $t$ sufficiently negative as well, giving a contradiction because $\gamma$ must remain homotopically nontrivial in $H_t$ since the flow is smooth. Alternately the proof of lemma \ref{2ends} implies at least two ends of $M_{-\infty}$ must be connected, ruling out this case.
$\medskip$

Now, since $M_t$ is eternal by the comparison principle it is noncompact so has at least one end, and if the number of ends is just one then we see from the above that $M_t$ must be topologically a plane. We discuss this case more in the concluding remarks immediately below.   \end{proof}

\section{Concluding remarks} First, we further discuss the case left from above that $M_t \simeq \R^2$. Under the finite total curvature type condition this set is nonempty, considering the bowl soliton, but there are some hints to suggest there are none under a finite total curvature assumption of the form ``there is some constant $C > 0$ so that $\int_{M_t} |A|^2 < C$ for all $t < 0$.'' With this assumption (or one that implies this) to proceed it seems natural to consider the blowdown of $M_t$. By this we recall one parabolically rescales $M_t$ about the origin by a sequence of scales tending $\lambda_i \to 0$. Denoting the resulting flows by $M^{\lambda_i}_t$, by the finite entropy assumption as in the section above one may employ Brakke's compactness theorem (considering the flows as Brakke flows) to extract a limiting Brakke flow which, by Huisken's monotonicity formula, must be a self shrinker. As discussed in \cite{IP}, its support will be smooth although the convergence may be with multiplicity greater than one and there are examples where this is the case: for instance, considering the catenoid as an eternal flow whose blowdown is the plane with multiplicity 2. 
$\medskip$

Since $M_t$ is mean convex the blowdown limit as well must be mean convex, so by the finite entropy assumption and \cite{CM} must either be a flat plane or round cylinder (or round sphere, which is ruled out by the noncompactness of $M_t$). In the later case, using that the total curvature is scale invariant we find a contradiction since the round cylinder has infinite total curvature (in terms of integral of $|A|^2$, not $|K|$ of course). The issue, at least from this line of attack, is if the blowdown is a plane. Indeed,  its blowdown limit cannot be a flat plane of multiplicity one -- if this were the case then by Huisken's monotonicity the flow $M_t$ itself must be a flat plane, contradicting that we assume $H > 0$ strictly along the flow. It seems that ruling out the case the blowdown limit is a plane of multiplicity 2 or higher though is more delicate, but to the author's knowledge the known examples of such flows diffeomorphic to the plane, such as the grim reaper cylinder and the $\Delta$--wings (see \cite{HIMW, BLT1}), have infinite total curvature. There are also more exotic convex eternal flows which aren't translators, constructed in \cite{BLT}. 
$\medskip$

Note that in the above no claims of uniqueness are made for the eternal solution up to the catenoid it flows out of -- it seems quite plausible that such a solution is unique though. For instance, it is shown in \cite{CCMS} that there was only one ancient rescaled mean curvature flow which laid to one side (a priori not even shrinker mean convex) of an asymptotically conical self shrinker, using the Merle--Zaag ODE lemma. There are a number of issues one needs to address though, for instance how well an eternal solution $M_t$ can be shown to converge back in time to its catenoid $C$ -- consider lemma 7.18 in \cite{CCMS}. We also mention that in \cite{MP} some uniqueness results for suitably symmetric eternal mean convex flows were shown -- see theorem 1.5 therein. In a nutshell, the assumptions quickly imply that such a flow must asymptote to a catenoid as $t \to -\infty$, and using uniqueness results for the curve shortening flow we can show an ``asymptotic uniqueness'' statement.
$\medskip$

As discussed already in the introduction, if we relax our hypotheses of mean convexity we then must include all minimal surfaces of finite topology, the space of which is comparatively quite large with varied topology. In this space there may also be topologically interesting translators of finite total curvature. By \cite{Khan} such translators must have asymptotically planar ends and with this in mind one could imagine, for instance, the existence of a ``sideways translating annulus'' of finite total curvature, although these have recently been conjectured to not exist \cite{HMW}. If we consider relaxing our assumption of finite total curvature type there are clearly weaker assumptions which will allow us to take a backwards limit and get something smooth, although possibly not of finite total curvature depending on the manner the assumptions were loosened which played a role in the argument. There is also the question of generalizing this result to higher dimensions, which is reasonable to ask because the reapernoid of \cite{MP} was constructed for all $n \geq 2$, although we point out that in a number of steps (for instance in taking the backwards limit) that we were in the surface case was strongly used.

\end{document}